\documentclass[10pt]{article}
\usepackage{fancyhdr,graphicx}
\usepackage{amsfonts}
\usepackage{amssymb}
\usepackage{amsthm}
\usepackage{newlfont}
\usepackage{amsmath} 
\usepackage[frame,arrow,curve,matrix]{xy}
\usepackage{subfigure}

\newtheorem{thm}{Theorem}[section]
\newtheorem{lem}[thm]{Lemma}

\graphicspath{{figures/}}

\usepackage{mathrsfs}  
\usepackage{subfig}
\usepackage{picinpar}
\begin{document}

\title{Unknotting twisted knots with Gauss diagram forbidden moves}

\author{Shudan Xue \quad Qingying Deng\footnote{Corresponding author. \newline {\em E-mail address:} 201921001184@smail.xtu.edu.cn (S. Xue), qingying@xtu.edu.cn (Q. Deng).} \\
{\footnotesize   \em  School of Mathematics and Computational Science, Xiangtan University, Xiangtan, Hunan 411105,
P. R. China}}

\date{ }

\maketitle

\begin{abstract}
Twisted knot theory, introduced by M.O. Bourgoin, is a generalization of virtual knot theory.
It is well-known that any virtual knot can be deformed into a trivial knot by a finite sequence of generalized Reidemeister moves and two forbidden moves $F1$ and $F2$. Similarly, we show that any twisted knot also can be deformed into a trivial knot or a trivial knot with a bar by a finite sequence of extended Reidemeister moves and three forbidden moves $T4$, $F1$ (or $F2$) and $F3$ (or $F4$) .
\end{abstract}

$\mathbf{Keywords:}$ Twisted link; Gauss code; Gauss diagram; Forbidden moves.

\vskip0.5cm

\section{Introduction}
\setlength{\parindent}{2em}
Virtual knot theory is a generalization of classical knot theory which is introduced by Kauffman \cite{L.H.K}.
As an appropriate device to describe finite type invariants \cite{T.F2}, Gauss diagram was introduced by Polyak and Viro \cite{M.O} in 1994. Kauffman \cite{L.H.K} developed fruitful theory of virtual knots. Decorated Gauss diagrams introduced by Fiedler \cite{T.F1,T.F2} are an efficient tool, and Mortier \cite{A.M} showed that a knot diagram can be fully recovered from its Gauss diagram and characterized the decorated Gauss diagram of closed braids.
Kwun, Nizami, Nazeer, Munir and Kang \cite{Y.A.W} proved that the Gauss diagram remains unchanged if a knot is mirrored, and is mirrored if the knot is reversed.
In 2001, Kanenobu \cite{T.K} and Nelson \cite{S.N} independently proved the same results, that is, any virtual knot can be deformed into a trivial knot by a finite sequence of generalized Reidemeister moves and forbidden moves $F1$ and $F2$. But the major difference lies in the proof method, the former uses virtual braid moves, the latter combines with Gauss diagrams.

In $2008$, M.O. Bourgoin \cite{M.O.B} generalized virtual links to twisted links.
Virtual links are regarded as twisted links. Recently, S. Kamada and N. Kamada discussed when two virtual
links are equivalent as twisted links, and gave a necessary and sufficient
condition for this to be the case in \cite{Kamada2020}.

Our main contribution in this regard is that we show that any twisted knot also can be deformed into a trivial knot or a trivial knot with a bar by a finite sequence of extended Reidemeister moves and three forbidden moves $T4$, $F1$ (or $F2$) and $F3$ (or $F4$).

The rest of this paper is organized as follows.
In section $2$, we introduce twisted knot theory and Gauss diagram.
In section $3$, we show the main Theorem for twisted knots.

\section{Twisted Knot Theory And Gauss Diagram}

\begin{figure}[!htbp]
  \centering
    \subfigure[$ classical \ knot$]{
  \includegraphics[width=0.19\textwidth]{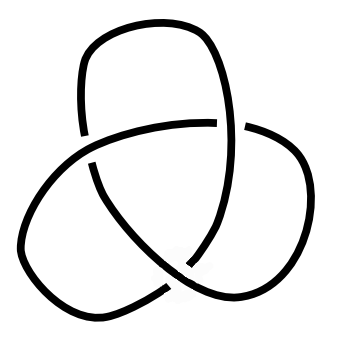}}  \
  \
   \subfigure[$  virtual \ knot$]{
  \includegraphics[width=0.20\textwidth]{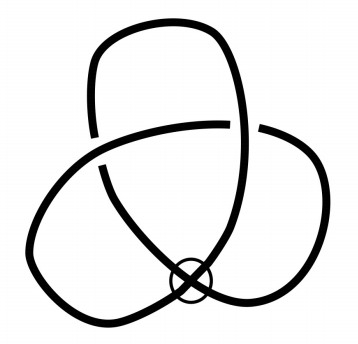}}  \
  \
  \subfigure[$  twisted \ knot$]{
  \includegraphics[width=0.20\textwidth]{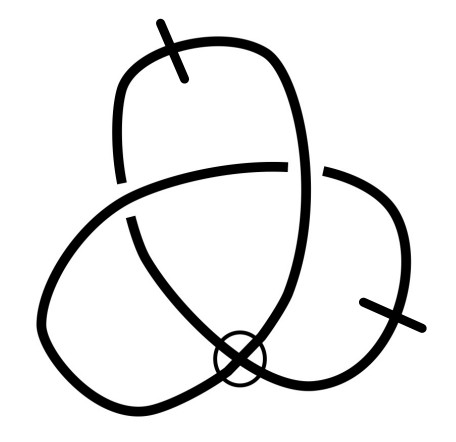}}  \
  \caption{Examples}\label{EX}
\end{figure}

Virtual knot theory is an extension of classical diagrammatic knot theory (see Figure \ref{EX}(a)).
A \emph{virtual link diagram} is a link diagram which may have virtual crossings, which are encircled crossings without over-under information (see Figure \ref{EX}(b)). A \emph{virtual link} is an equivalence class of virtual link diagrams by \emph{(classical) Reidemeister moves} $R1$, $R2$, $R3$ and \emph{virtual Reidemeister moves} $V1$, $V2$, $V3$, $V4$ in Figure \ref{R}. All of these are called \emph{generalized Reidemeister moves}.

 \begin{figure}[!htbp]
  \centering
  \includegraphics[width=1\textwidth]{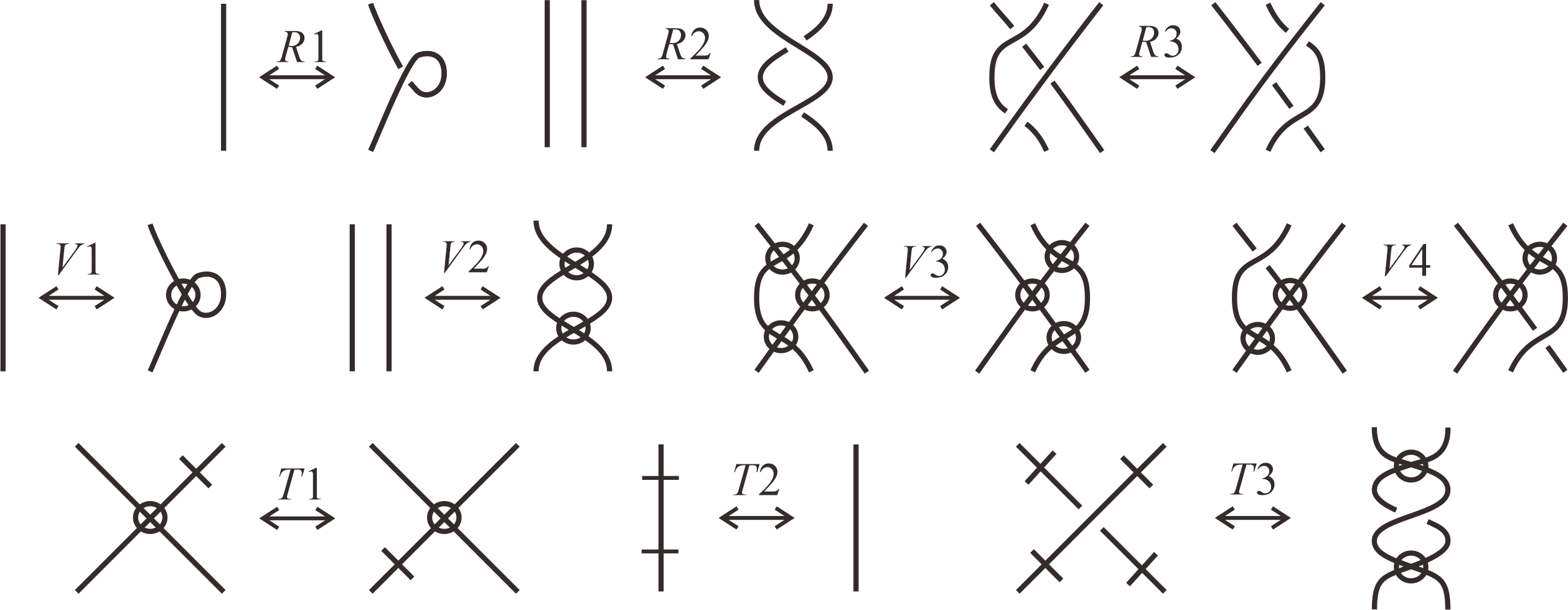}\\
  \caption{Extended Reidemeister moves.}\label{R}
\end{figure}

In $2008$, M.O. Bourgoin \cite{M.O.B} generalized virtual links to twisted links.
A \emph{twisted link diagram} is a virtual link diagram which may have some bars (see Figure \ref{EX}(c)).
A \emph{twisted link} is an equivalence class of twisted link diagrams by \emph{(classical) Reidemeister moves} $R1$, $R2$, $R3$, \emph{virtual Reidemeister moves} $V1$, $V2$, $V3$, $V4$ and \emph{twisted Reidemeister moves} $T1$, $T2$, $T3$ in Figure \ref{R}. All of these are called \emph{extended Reidemeister moves}. The reader may consult \cite{M.O.B} for more details about twisted knot theory.

\begin{figure}[!htbp]
  \centering
  \includegraphics[width=0.30\textwidth]{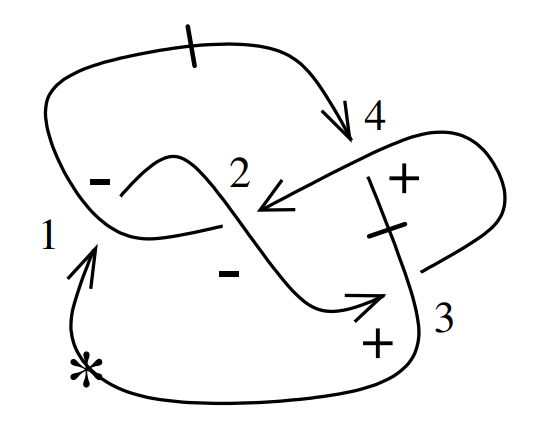}\\
  \caption{Gauss code : U1-O2-U3+O4+U2-O1-bU4+bO3+}\label{GC}
\end{figure}

A \emph{Gauss diagram} is a diagram that is a diagrammatic representation of the Gauss code of a knot. A central question about Gauss diagram is to determine a Gauss code for a given oriented knot diagram. For our purpose, let $K$ be a given twisted oriented knot diagram, then its Gauss code is obtained as follows. First, we associate each classical crossing with a digital labeling (such as 1,2, $\cdots$) and also indicate the crossing type ($+$ or $-$ for positive or negative crossing, respectively). Addition care needs to be taken when there is a bar on an arc between classical crossing points. If a bar on an arc between classical crossing points is met, then you need record the labelling $b$. Then we begin walking along the diagram by choosing a basepoint on the knot diagram, and record the labeling of the crossings and bars \cite{Y.A.W}. For example, if you go under crossing 1 whose sign is $+$ then you will record it as $U1+$, then cross a bar and go under crossing 2 whose sign is $-$ then you will record it as $U1+bU2-$. An example of a twisted oriented knot diagram $K$ and its Gauss code is shown in Figure \ref{GC}.

\begin{figure}[!htbp]
  \centering
    \subfigure[Twisted knot diagram $K$]{
  \includegraphics[width=0.30\textwidth]{Gauss.png}}  \
   \subfigure[Gauss diagram of $K$]{
  \includegraphics[width=0.28\textwidth]{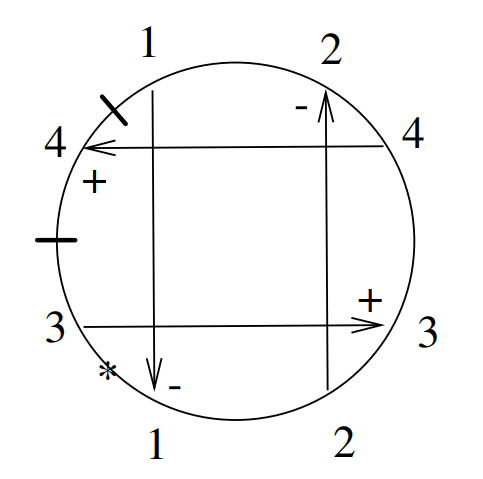}}  \
  \caption{Example}\label{GD}
\end{figure}

If a Gauss code for an oriented knot diagram is given, we construct its Gauss diagram by placing an oriented circle with a basepoint chosen on the circle and then recording the labeling for the crossings and bars on the circle according to the order of the Gauss code whenever walking along the circle counterclockwise. To record the information of the crossing type, now we add chords between two same labelings on the circle and orient each chord from overcrossing site to undercrossing site. Mark each chord with $+$ or $-$ according to the sign of the corresponding crossing in the Gauss code \cite{Y.A.W}. We use $\varepsilon$ to denote ``$\pm$". The resulting graph is the Gauss diagram corresponding to Gauss code for an oriented knot diagram. An example of a twisted oriented knot diagram $K$ and its Gauss diagram is shown in Figure \ref{GD}.

It should be pointed out that not every Gauss diagram without bars corresponds to a classical knot, but these Gauss diagrams can correspond to virtual knots \cite{L.H.K, S.N, M.M.O}.

We present Gauss diagrams corresponding to generalized Reidemeister moves (see Figure \ref{GR}). Note that there are several cases of move R3 depending on the orientations of the strands; we only depict one.
By definition, Gauss diagram remains unchanged under virtual Reidemeister moves.

\begin{figure}[!htbp]
  \centering
   \subfigure[$R1$]{
  \includegraphics[width=0.23\textwidth]{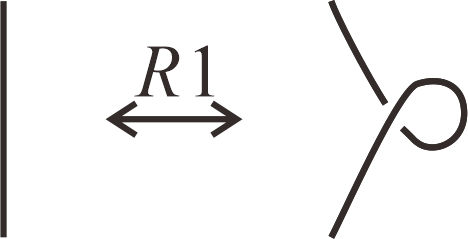}} ~~~~~
      \subfigure[Gauss diagram of $R1$]{
  \includegraphics[width=0.40\textwidth]{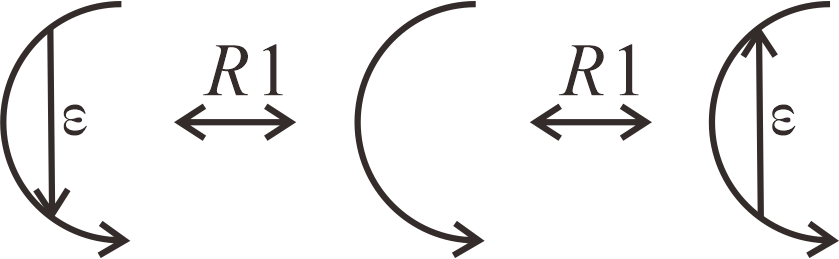}}  \\
 \subfigure[$R2$]{
  \includegraphics[width=0.27\textwidth]{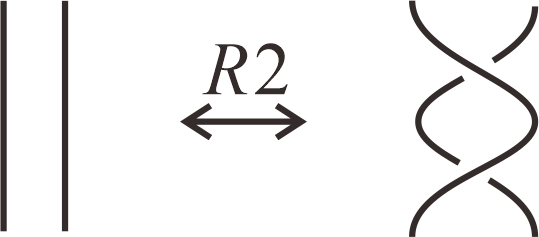}} ~~~~~
        \subfigure[Gauss diagram of $R2$]{
  \includegraphics[width=0.45\textwidth]{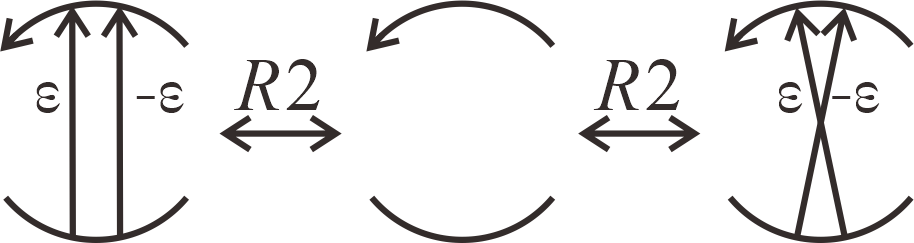}}  \\
   \subfigure[$R3$]{
  \includegraphics[width=0.27\textwidth]{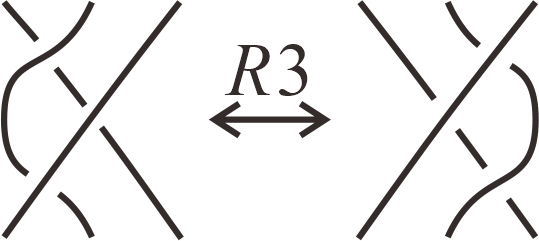}}  ~~~~~
      \subfigure[Gauss diagram of $R3$]{
  \includegraphics[width=0.40\textwidth]{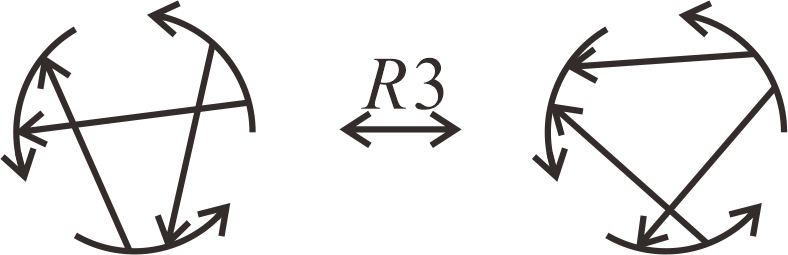}}  \\
  \caption{Gauss diagram of classical Reidemeister moves}\label{GR}
\end{figure}

Goussarov, Polyak and Viro observed that there are two forbidden moves $F1$ and $F2$ (see Figure \ref{FM1}) on virtual knot diagrams. Any virtual knot diagram can be deformed into a trivial knot diagram when these two moves are allowed \cite{T.K, S.N} (see Theorem \ref{virtual}).
It is straightforward to check that the signs on the chords of Gauss diagrams of $F1$ and $F2$ are arbitrary.

\begin{figure}[!htbp]
  \centering
    \subfigure[$F1$]{
  \includegraphics[width=0.28\textwidth]{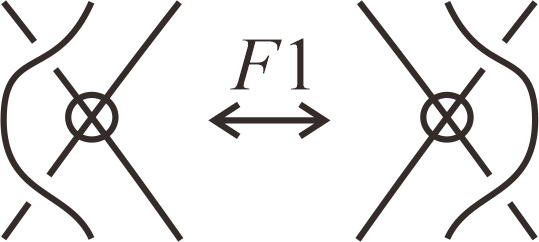}} ~~~~~
  \subfigure[Gauss diagram of $F1$]{
  \includegraphics[width=0.41\textwidth]{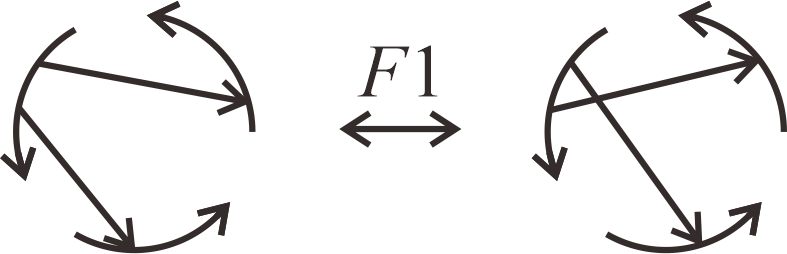}} \\
    \subfigure[$F2$]{
  \includegraphics[width=0.28\textwidth]{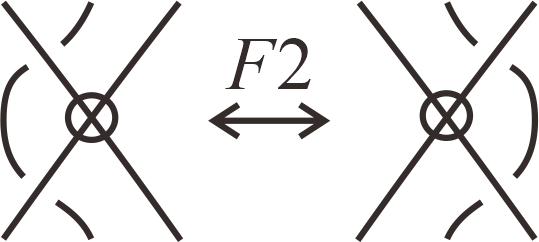}} ~~~~~
      \subfigure[Gauss diagram of $F2$]{
  \includegraphics[width=0.41\textwidth]{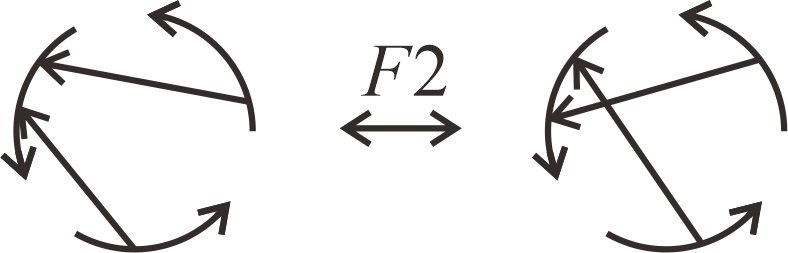}} \\
  \caption{Forbidden moves $F1$ and $F2$}\label{FM1}
\end{figure}

\begin{figure}[!htbp]
  \centering
  \includegraphics[width=0.9\textwidth]{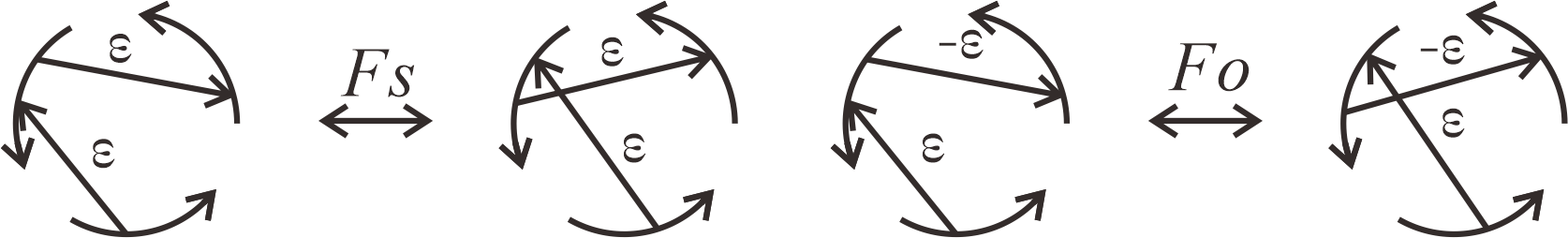}
  \caption{$Fs$ and $Fo$}\label{Fso}
\end{figure}

Nelson proved Theorem \ref{virtual} by the following two moves (see Figure \ref{Fso}): an arrowhead past an arrowtail of the same sign (move $Fs$) or past an arrowtail of the opposite sign (move $Fo$) using (classical) Reidemeister moves and forbidden moves $F1$ and $F2$. Note that moves $Fs$ and $Fo$ also are forbidden in virtual knot theory, so we call moves $Fs$ and $Fo$ forbidden moves in the following statement.

\begin{thm}(\cite{S.N}, Theorem 1)\label{virtual}
Any Gauss diagram of virtual knot can be changed into Gauss diagram of any other virtual knot by a sequence of moves of types $R1$, $R2$, $R3$, $F1$ and $F2$.
\end{thm}

\section{The Main Theorem}

One can also consider Gauss diagram in twisted knot theory. We present Gauss diagrams corresponding to twisted Reidemeister moves (see Figure \ref{TGD}). By definition, Gauss diagram remains unchanged under twisted Reidemeister move $T1$.

\begin{figure}[!htbp]
  \centering
      \subfigure[$T2$]{
  \includegraphics[width=0.22\textwidth]{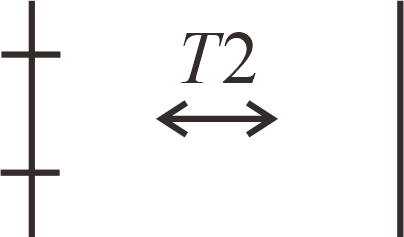}}  ~~~~~
   \subfigure[Gauss diagram of $T2$]{
  \includegraphics[width=0.34\textwidth]{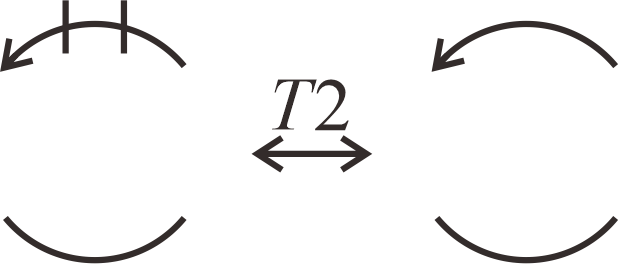}}  \\
  \subfigure[$T3$]{
  \includegraphics[width=0.36\textwidth]{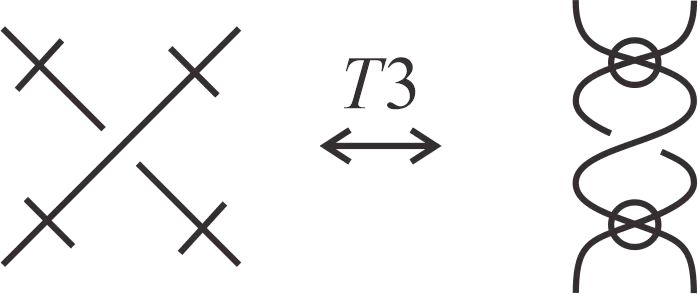}}  ~~~~~
        \subfigure[Gauss diagram of $T3$]{
  \includegraphics[width=0.35\textwidth]{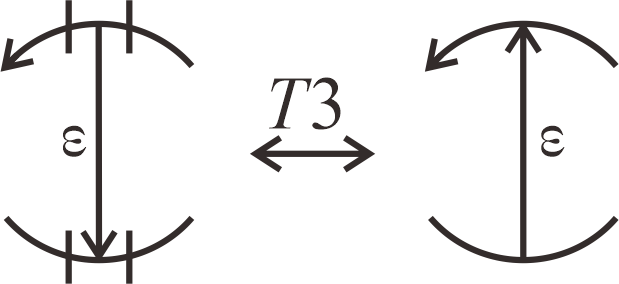}}  \\
  \caption{Gauss diagram corresponding to twisted Reidemeidter moves}\label{TGD}
\end{figure}

In consideration of the case with bars, we naturally think of forbidden moves $F3$ and $F4$ (see Figure \ref{FM2}).

\begin{figure}[!htbp]
  \centering
    \subfigure[$F3$]{
  \includegraphics[width=0.30\textwidth]{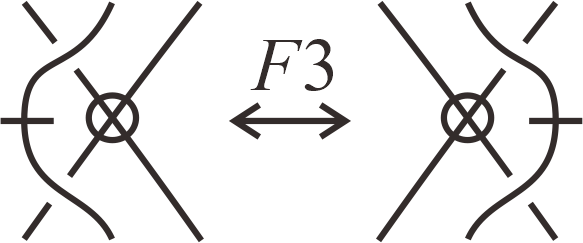}} ~~~~~
  \subfigure[Gauss diagram of $F3$]{
  \includegraphics[width=0.40\textwidth]{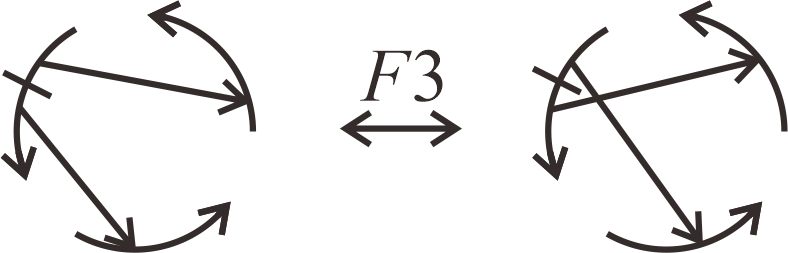}} \\
\subfigure[$F4$]{
  \includegraphics[width=0.30\textwidth]{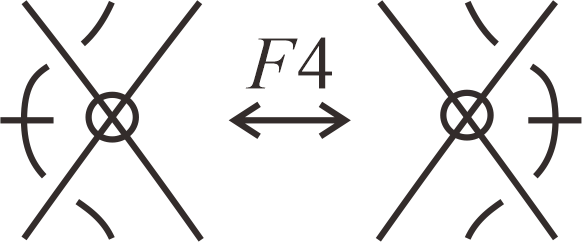}} ~~~~~
      \subfigure[Gauss diagram of $F4$]{
  \includegraphics[width=0.40\textwidth]{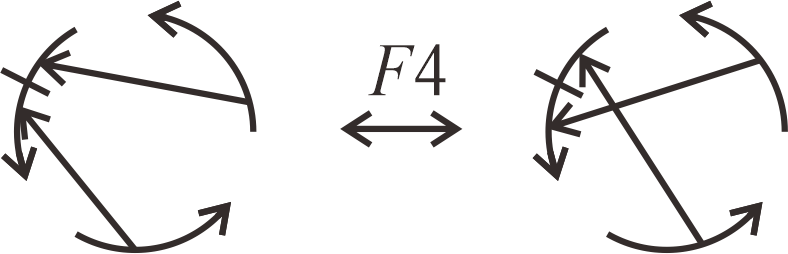}} \\
  \caption{Forbidden moves $F3$ and $F4$}\label{FM2}
\end{figure}

It is easily checked that forbidden moves $F1$ and $F2$ are equivalent by using Gauss diagram corresponding to moves $T2$ and $T3$. Figure \ref{P1}(a) illustrates that a forbidden move $F1$ is realized by moves $F2$, $T2$ and $T3$. Similarly, a forbidden move $F2$ is realized by moves $F1$, $T2$ and $T3$.

In virtual knot theory, if one of the two forbidden moves of Figure \ref{FM1} is allowed, move $F1$, which contains an over arc and one virtual crossing, we arrive at welded knot theory of Fenn, Rim\'{a}nyi and Rourke \cite{R.F}. But in twisted knot theory, forbidden move $F1$ and $F2$ are equivalent. We can't extend welded knot theory to twisted case.

It is easily checked that forbidden move $F3$ and $F4$ are equivalent (see Figure \ref{P1}(b)).

\begin{figure}[!htbp]
  \centering
    \subfigure[$F2 \rightarrow F1$]{
  \includegraphics[width=0.8\textwidth]{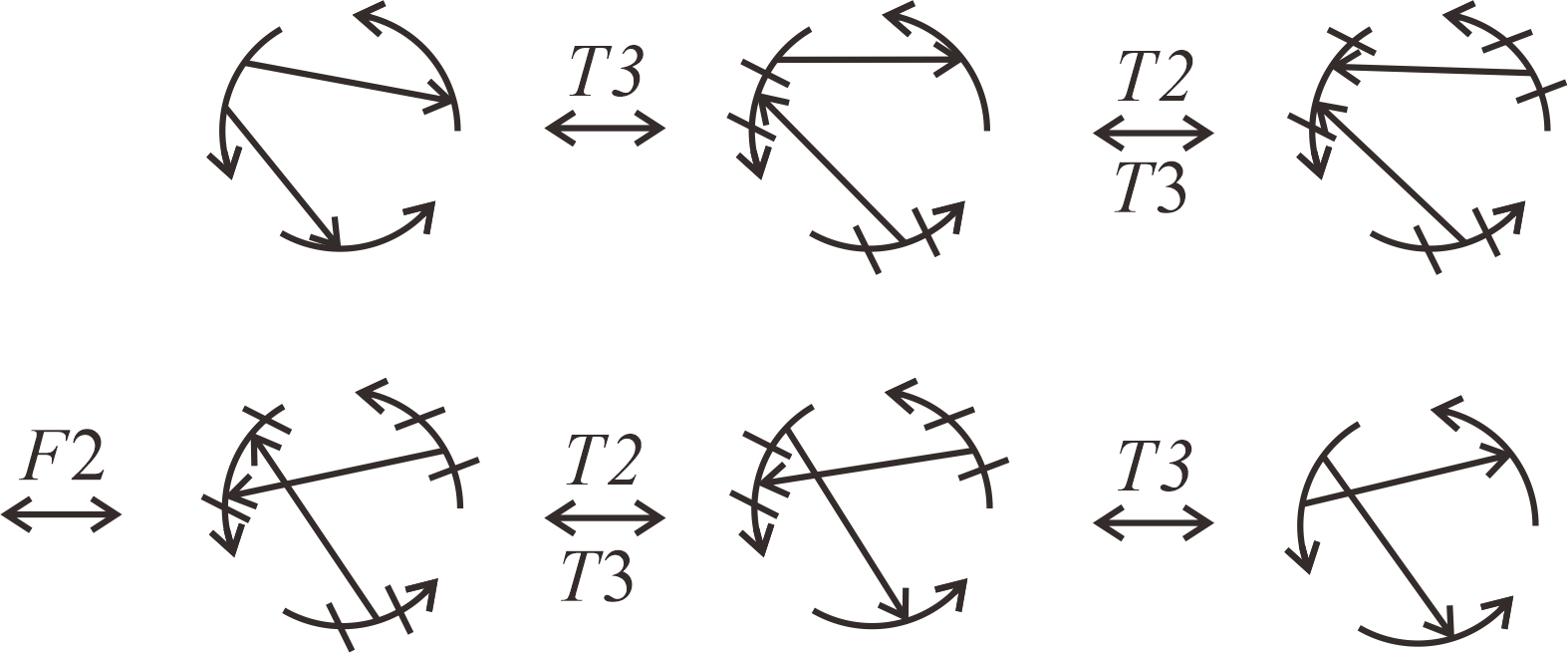}} \\
   \subfigure[$F4 \rightarrow F3$]{
  \includegraphics[width=0.8\textwidth]{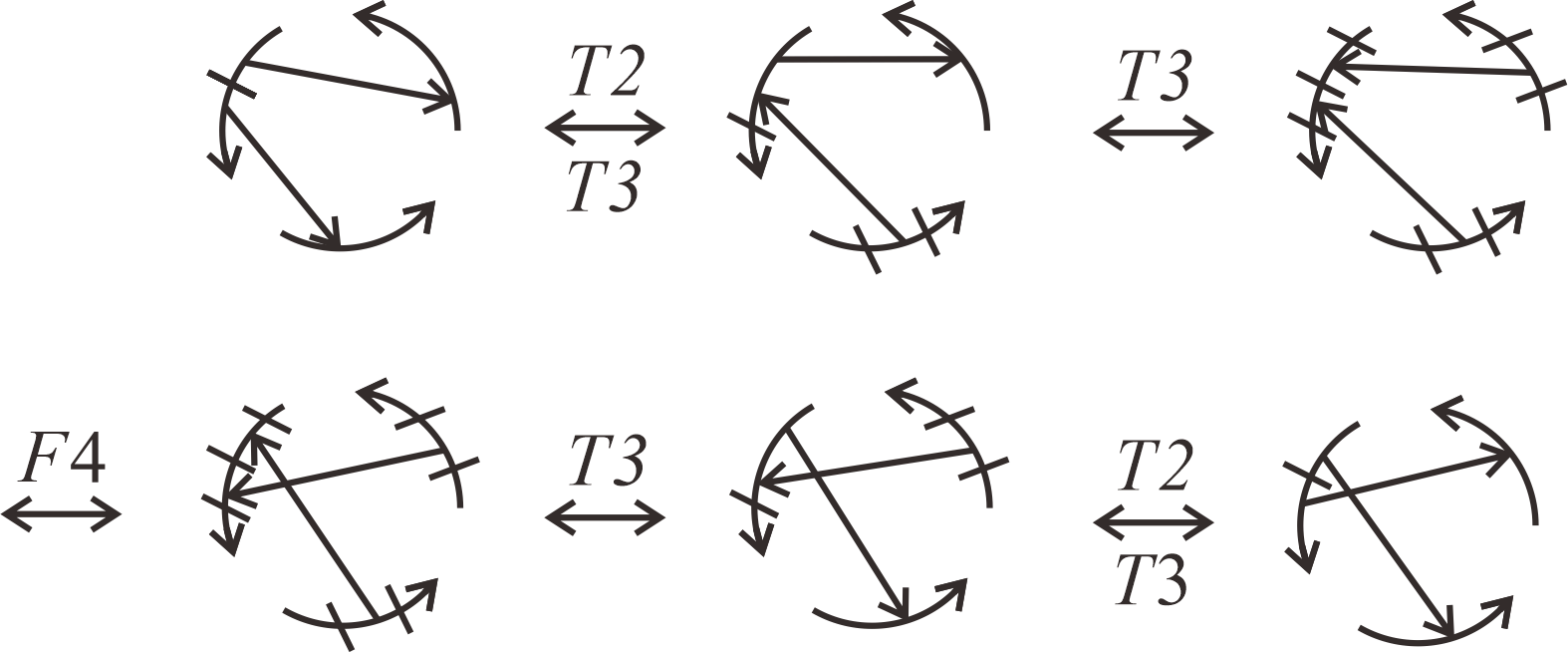}}  \\
  \caption{Preparation for Theorem \ref{twisted}}\label{P1}
\end{figure}

\begin{figure}[!htbp]
  \centering
    \subfigure[$T4$]{
  \includegraphics[width=0.27\textwidth]{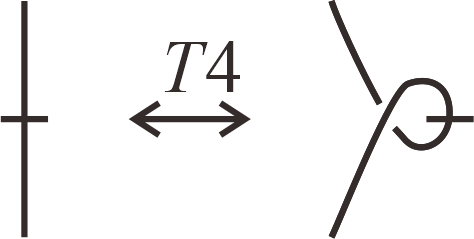}}~~~~~
  \subfigure[Gauss diagram of $T4$]{
  \includegraphics[width=0.46\textwidth]{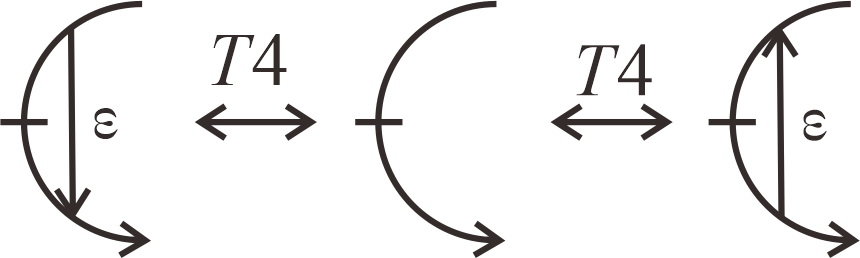}}\\
  \caption{Forbidden moves $T4$}\label{P2}
\end{figure}

To obtain the main result of our paper, we introduce forbidden move $T4$ (Figure \ref{P2}(a)) that corresponds to $R1$ with a bar, which means that we can delete or add a curl with bar.

Forbidden moves $Fs$ and $Fo$ are moves that involve an arrowhead with either sign past an adjacent arrowtail with either sign. In twisted knot theory, we naturally consider whether an arrowhead with either sign can go over a bar and then an adjacent arrowtail with either sign, that is, do the following moves $Fu$ and $Fv$ (Figure \ref{Fuv}) hold, and under what conditions?

\begin{figure}[!htbp]
  \centering
  \includegraphics[width=0.8\textwidth]{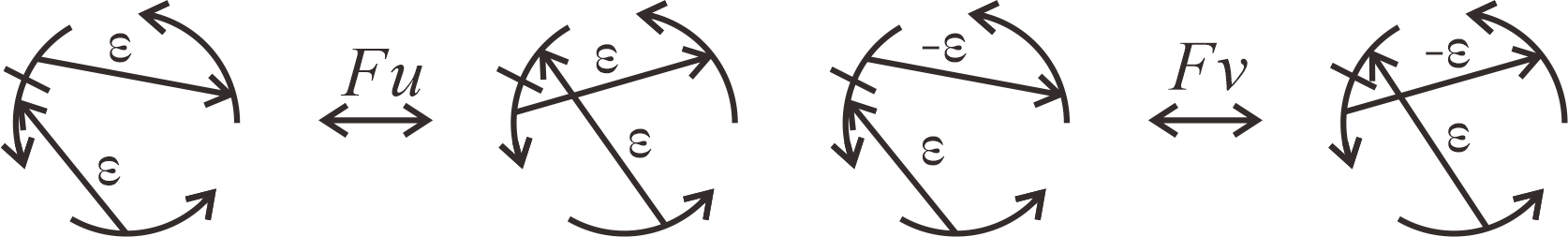}
  \caption{$Fu$ and $Fv$}\label{Fuv}
\end{figure}

\begin{lem}\label{L}
If forbidden moves $T4$, $F1$ (or $F2$) and $F3$ (or $F4$) are allowed in twisted knot theory, then moves $Fu$ and $Fv$ hold.
\end{lem}

\begin{proof}

Figure \ref{Fu} illustrates that move $Fu$ is realized by $R1$, $T2$, $T3$, $T4$, $F1$, $F3$, $F4$ and $Fs$.

Figure \ref{Fv} illustrates that move $Fv$ is realized by $R1$, $T2$, $T3$, $T4$, $F1$, $F3$, $F4$, $Fs$ and $Fo$.

\begin{figure}[!htbp]
  \centering
  \includegraphics[width=0.8\textwidth]{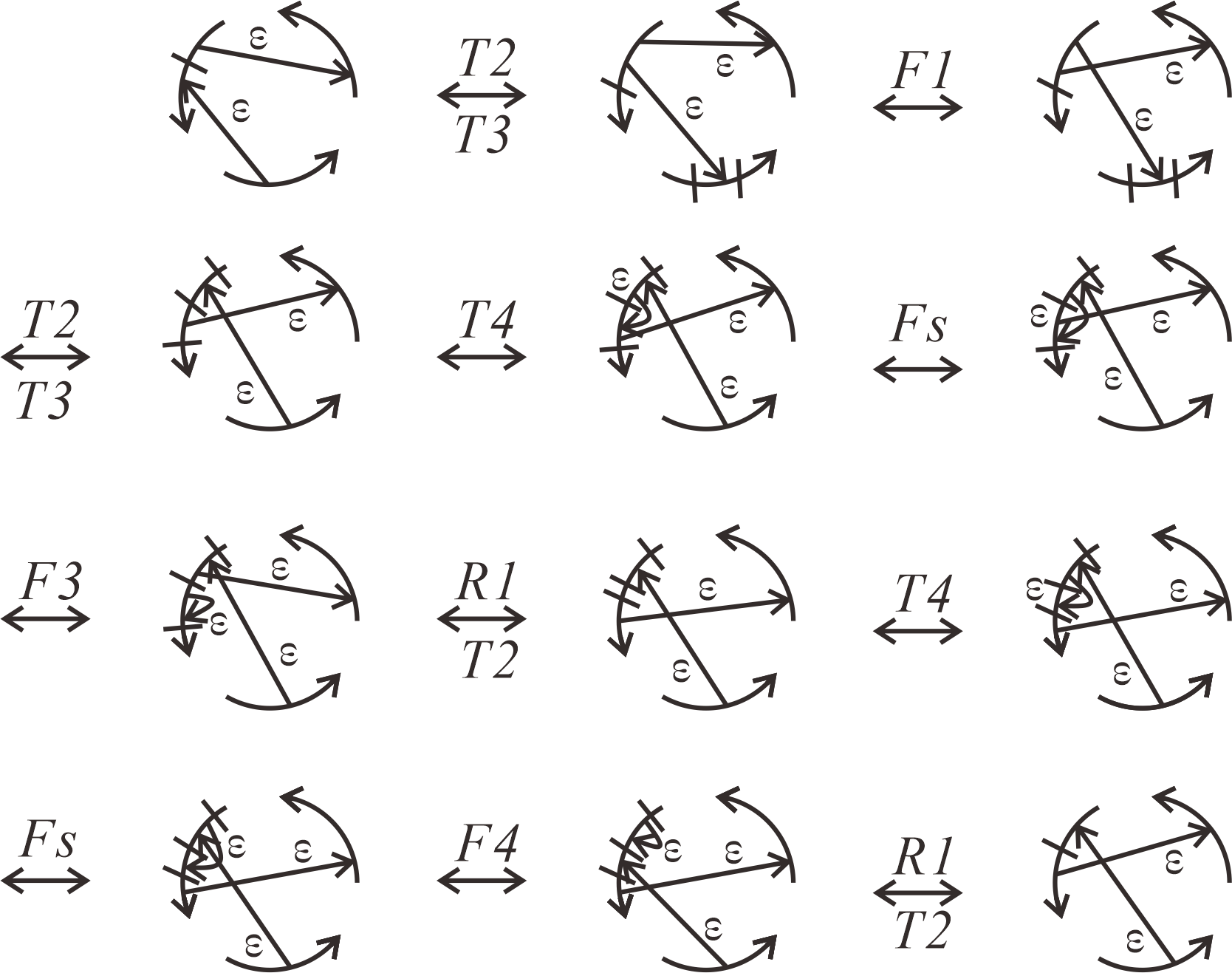}\\
  \caption{Proof of $Fu$}\label{Fu}
\end{figure}

\begin{figure}[!htbp]
  \centering
  \includegraphics[width=0.8\textwidth]{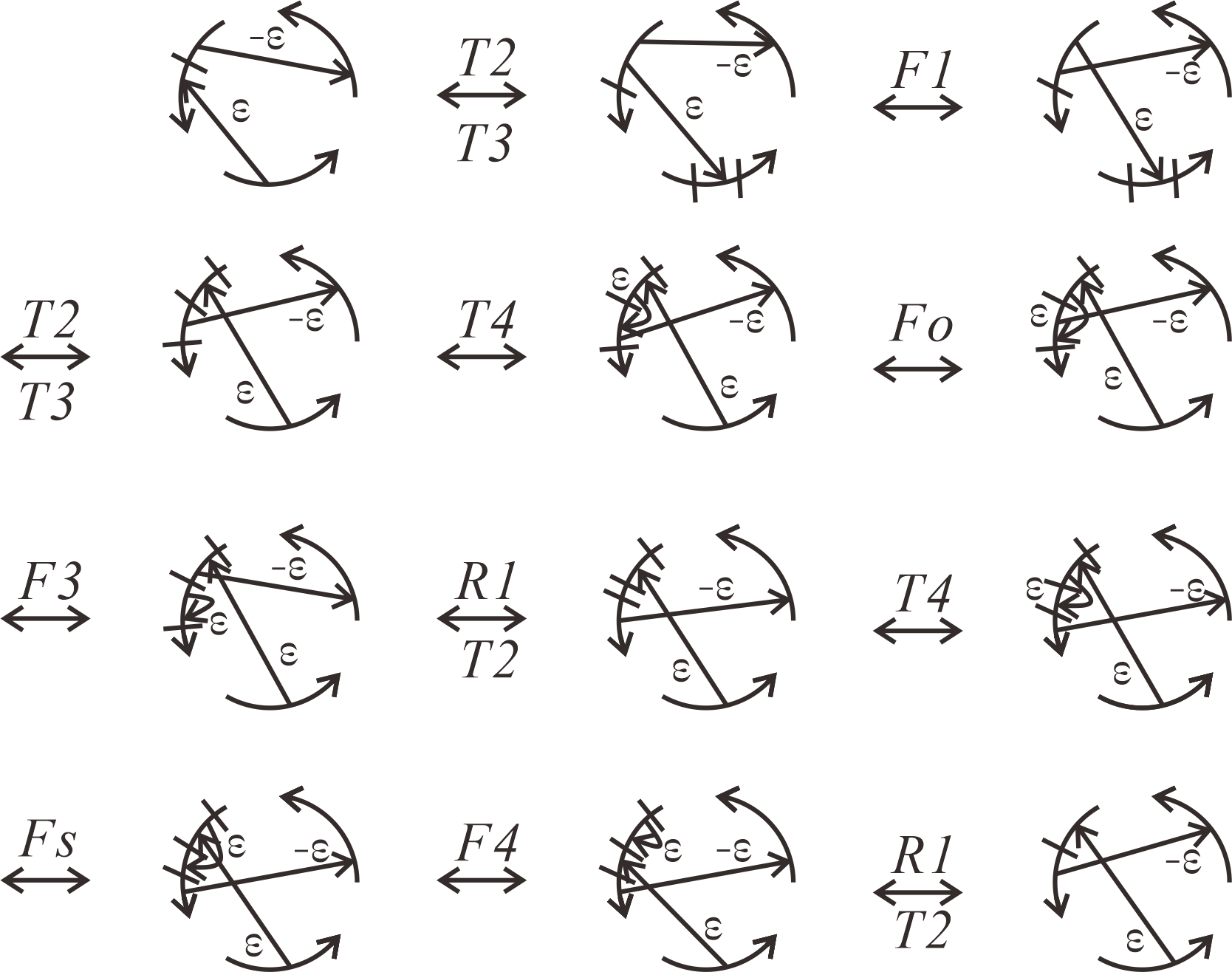}\\
  \caption{Proof of $Fv$}\label{Fv}
\end{figure}

\end{proof}

\begin{thm}\label{twisted}
Any Gauss diagram of twisted knot can be changed into Gauss diagram of a trivial knot (with a bar) by a sequence of moves of types $R1$, $R2$, $R3$, $T2$, $T3$ and forbidden moves $T4$, $F1$ (or $F2$) and $F3$ (or $F4$).
\end{thm}

\begin{proof}

We observe that forbidden move $F1$ ($F3$) allows us to move an arrowhead with either sign past (with bar) an adjacent arrowhead with either sign without condition on the tails of these respective arrows, and move $F2$ ($F4$) let us do the same with arrowtails (with bar).

It should be pointed out that forbidden moves $Fs$ and $Fo$ ($Fu$ and $Fv$) allow us to move an arrowhead of either sign past (with bar) an arrowtail of either sign in the same manner, then we can simply rearrange the arrows in a given diagram at will.

Now, it suffices to show that Gauss diagram of any twisted knot may be converted to Gauss diagram of a trivial knot (with a bar), we simply use forbidden moves $F1$, $F2$, $F3$, $F4$, $Fo$, $Fs$, $Fu$ and $Fv$ to rearrange the arrows. If we have extra arrows of either sign, we can use forbidden moves $F1$, $F2$, $F3$, $F4$, $Fo$, $Fs$, $Fu$ and $Fv$ to move unwanted arrows into specific position to be removed by type $R1$ and $T4$ moves. Then any twisted knot can be unknotted by this technique.
\end{proof}

\section{Acknowledgements}
This research is supported by the National Natural Science Foundation of China (No. 12001464),
the General Research Project of Hunan Provincial Department of Education (NO. 20C1766),
the Natural Science Foundation of Hunan Province (No. 2022JJ40418),
the Doctor's Funds of Xiangtan University (No. 09KZ$|$KZ08069),
and the Hu Xiang Gao Ceng Ci Ren Cai Ju Jiao Gong Cheng-Chuang Xin Ren Cai (No. 2019RS1057).

\end{document}